\documentclass{amsart}

\usepackage[utf8]{inputenc}

\usepackage{amsmath}
\usepackage{amsthm}
\usepackage{amssymb}
\usepackage{mathtools}
\usepackage{bm}		
\usepackage{mathdots}
\usepackage{framed}
\usepackage{hyperref}
\usepackage[capitalize]{cleveref}
\usepackage{array}
\usepackage[all,cmtip]{xy}

\usepackage{dsfont}

\usepackage{xcolor}

\usepackage{pdfpages}

\usepackage[pass]{geometry}

\def\N{{\mathbb N}}
\def\Z{{\mathbb Z}}

\def\C{{\mathbb C}}

\def\A{{\mathbb A}}

\def\cO{\mathcal{O}}

\def\cX{\mathcal{X}}
\def\cY{\mathcal{Y}}
\def\cZ{\mathcal{Z}}

\def\fa{\mathfrak{a}}

\def\fm{\mathfrak{m}}
\def\fn{\mathfrak{n}}
\def\fp{\mathfrak{p}}

\def\.{\cdot}
\let\circum\^
\def\^{\widehat}
\def\~{\widetilde}

\def\({\left(}
\def\){\right)}

\def\*{{}^*}

\newcommand{\cotimes}[1][]{\: \widehat{\otimes}_{#1} \:} 

\renewcommand{\and}{ \ \ \text{ and } \ \ }

\def\sm{\mathrm{sm}}

\DeclareMathOperator{\Spec} {Spec}
\DeclareMathOperator{\Spf} {Spf}

\DeclareMathOperator{\id} {id}

\DeclareMathOperator{\Hom} {Hom}

\DeclareMathOperator{\pr} {pr}

\DeclareMathOperator{\Der}{Der}

\def\isom{\simeq} 

\def\embdim{\mathrm{edim}}
\def\embcodim{\mathrm{ecodim}}

\theoremstyle{plain}

\newtheorem{introtheorem}{Theorem}

\crefname{introtheorem}{Theorem}{Theorems}

\newtheorem{theorem}{Theorem}[section]
\newtheorem*{theorem*}{Theorem}
\newtheorem{proposition}[theorem]{Proposition}
\newtheorem{lemma}[theorem]{Lemma}
\newtheorem{corollary}[theorem]{Corollary}

\theoremstyle{definition}

\newtheorem{definition}[theorem]{Definition}

\theoremstyle{remark}

\newtheorem{remark}[theorem]{Remark}
\newtheorem{example}[theorem]{Example}
\newtheorem{question}[theorem]{Question}

\numberwithin{figure}{section}

\numberwithin{equation}{section}

\usepackage[
backend=biber,
style=numeric,
giveninits=true
]{biblatex}

\DefineBibliographyStrings{english}{%
	backrefpage = {cited on page},
	backrefpages = {cited on pages},
}

\DeclareSourcemap{
  \maps[datatype=bibtex]{
    \map{
        \step[fieldset=doi,null]
        \step[fieldset=url,null]
        \step[fieldset=isbn,null]
        \step[fieldset=issn,null]        
        \step[fieldset=note,null]
        \step[fieldset=pagetotal,null]
    }
  }
}

\begin{document}


\title{On the formal neighborhood of a degenerate arc}

\author{Christopher Chiu}
  \address[C.\ Chiu]{%
    Department of Mathematics\\
    KU Leuven\\
    Celestijnenlaan 200b - bus 2400\\
    3001 Leuven (Belgium)%
  }
  \email{christopherheng.chiu@kuleuven.be}
  
\author{Herwig Hauser}
  \address[H.\ Hauser]{%
    Fakult\"at f\"ur Mathematik\\
    Universit\"at Wien\\
    Oskar-Morgenstern-Platz 1\\
    A-1090 Wien (\"Osterreich)%
  }
  \email{herwig.hauser@univie.ac.at}


\thanks{The first author was supported by a postdoc fellowship (12AZ524N) of the Research Foundation - Flanders (FWO). The second author was partially supported by the Austrian Science Fund (FWF) through projects P31338 and P34765}
  
\begin{abstract}

 We prove a result describing the structure of the formal neighborhoods of certain arcs in the arc space of an algebraic variety which are completely contained in the singular locus. In particular, we provide a precise formulation of the intuitive statement that constant arcs centered in the singular locus are the most singular points of the arc space.

\end{abstract}

\maketitle


Let $X$ be a scheme of finite type over a field $k$. The arc space $X_\infty$ of $X$ is the $k$-scheme whose points parametrize formal arcs on $X$. In particular, for each field extension $K \supset k$, the $K$-points of $X_\infty$ correspond to morphisms
\[
    \Spec K[[t]] \to X.
\]
Since the work of Nash, the topological structure of $X_\infty$ has been known to encode deep properties about the singularities of $X$. More recently, the singularities of $X_\infty$ itself have been an object of study. One of the first main results in this direction is the theorem of Drinfeld, Grinberg and Kazhdan. It roughly says that, even though the arc space $X_\infty$ is infinite-dimensional, its singularities at arcs not contained in the non-smooth locus $X \setminus X_\sm$ are encoded in a finite-dimensional model. 

\begin{theorem*}[\protect{\cite{Dri}}]
    Let $X$ be a scheme of finite type over a field $k$. Let $\alpha \in X_\infty(k)$ be such that $\alpha \notin (X \setminus X_\sm)_\infty$. Then there exists a scheme $Y$ of finite type over $k$, a point $y \in Y(k)$ and an isomorphism of formal neighborhoods
    \[
        \^{(X_\infty)}_\alpha \isom \^Y_y \times \Spf k[[z_i \mid i\in I]].
    \]
\end{theorem*}

Here $\^X_x = \Spf \^{\cO_{X,x}}$ denotes the formal neighborhood of the scheme $X$ at $x\in X$. For further development see \cite{BNS16,BS17b,Dri20,HW21}.

In this article we are interested in the case where the main assumption of the Drinfeld--Grinberg--Kazhdan theorem is not fulfilled, that is, when the arc $\alpha$ is contained in $X \setminus X_\sm$. Such arcs are sometimes called \emph{degenerate}. In \cite{BS17} it is proven that the existence of a finite-dimensional model fails for the constant arc centered at $0$ on the variety $X = \Spec k[x,y]/(x^2 + y^2)$ over a non-algebraically closed field $k$. This failure was verified for all degenerate arcs on an arbitrary scheme of finite type over a perfect field in \cite{CdFD22}, showing that the assumptions in the Drinfeld--Grinberg--Kazhdan theorem are optimal. The approach in \cite{CdFD22} was to characterize nondegenerate arcs by means of the finiteness of embedding codimension, an invariant of singularities.

Our main result gives a description of the formal neighborhood of constant arcs in the case of characteristic $0$. We say that a formal $k$-scheme $\cX$ admits a smooth factor if $\cX \isom \cZ \times \Spf k[[z]]$, where the product is taken in the category of formal $k$-schemes.

\begin{introtheorem}
    \label{t:main-thm}
    Let $X$ be a scheme over a field $k$ of characteristic $0$. Let $x \in X(k)$ and for $n \in \N\cup \{\infty\}$ let $\alpha_x \in X_n$ be the constant $n$-jet resp.\ arc centered at $x$ (see \cref{s:jets-arcs}). Then $\^X_x$ admits a smooth factor if and only if $\^{(X_n)}_{\alpha_x}$ admits a smooth factor.
\end{introtheorem}

In particular, this partially recovers the aforementioned result of \cite{CdFD22} for arcs fully contained in the respective isosingular locus, see \cref{p:not-dgk}. We should emphasize that the statement of \cref{t:main-thm} holds only for constant arcs. In fact, the formal neighborhood of any nonconstant arc admits a smooth factor. This suggests that, in some sense, the worst singularities of the arc space are located at its constant arcs. A precise statement is mentioned in \cref{p:comparison}.

\subsection*{Acknowledgements}
We are very grateful to the organizers of the Mini workshop of singularities 2023 at the University of Tokyo for their hospitality. We would also like to thank Nero Budur for the useful discussions and comments.

%
%

\section{Jet and arc spaces}
\label{s:jets-arcs}

In this section we briefly review some of the basic properties of jet and arc spaces we will use in this paper. For a more comprehensive treatment we refer the reader to some of the excellent surveys such as \cite{Ish07,dF18}.

Throughout this section we will fix a field $k$. Let $X$ be a scheme over $k$. The $n$-th jet space $X_n$ of $X$ is the $k$-scheme representing the $n$-th jet functor of $X$
\[
    Y \mapsto \Hom_k(Y \times_k \Spec k[t]/(t^{n+1}),X).
\]
In particular, for each field extension $K \supset k$, the $K$-points of $X_n$ are in bijection with $n$-jets on $X$; that is, morphisms of the form
\[
    \alpha_n \colon \Spec K[t]/(t^{n+1}) \to X.
\]
Note that $X_0 = X$ and the first jet space $X_1$ is just the total space of the Zariski tangent bundle of $X$.

For each $m\geq n$ we have induced morphisms $\pi_{m,n} \colon X_m \to X_n$ forming an inverse system. As all $\pi_{m,n}$'s are affine, the inverse limit exists in the category of $k$-schemes and is called the arc space $X_\infty$ of $X$. By definition, $X_\infty$ represents the functor
\[
    Y \mapsto \Hom_k(\varinjlim_n Y \times \Spec k[t]/(t^{n+1}),X).
\]
In particular, the $K$-points of $X_\infty$ are in bijection with arcs on $X$
\[
    \alpha \colon \Spec K[[t]] \to X.
\]
The arc space comes attached with induced projections $\pi_n \colon X_\infty \to X_n$. Write $\eta$ for the generic point and $0$ for the unique closed point of $\Spec K[[t]]$. Then $\pi_0 \colon X_\infty \to X$ is given by $\alpha \mapsto \alpha(0)$, where we identify a $K$-point of $X_\infty$ with its corresponding arc. We say that $\alpha$ is centered at $\alpha(0)$.

Let $n \in \N \cup \{\infty\}$. We define a section $\iota_n \colon X \to X_n$ of $\pi_{n,0} \colon X_n \to X$ as follows. For each $k$-scheme $S$ and $S$-point $x$, the image of $x$ is given as the $S$-point of $X_n$ corresponding to the composition
\[
    \xymatrix{ S \times_k \Spec k[t]/(t^{n+1}) \ar[r]^-{\pr_1} & S \ar[r]^-x & X,}
\]
if $n\in \N$; and for $n=\infty$
\[
    \xymatrix{ \varinjlim_n S \times_k \Spec k[t]/(t^{n+1}) \ar[r]^-{\pr_1} & S \ar[r]^-x & X.}
\]
If $x \in X$ with residue field $k_x$, then $\alpha_x := \iota_n(x) \in X_n(k_x)$ is called the \emph{constant jet} or \emph{constant arc} centered at $x$. Note that the constant arc centered at $x$ satisfies $\alpha_x(0) = \alpha_x(\eta) = x$. If $x \in X(k)$ is a closed point, then $\alpha_x$ is the unique point of $X_\infty$ with this property; in general, $\alpha_x$ is only maximal with respect to specialization.

\begin{example}
    Consider the arc $\alpha$ on $\A^2$ given by
    \[
        k[x,y] \to k((x))[[t]],\: x\mapsto x + t, y \mapsto 0.
    \]
    Then $\alpha$ is a nonconstant arc with $\alpha(0) = \alpha(\eta)$.
\end{example}

Jet and arc spaces are functorial with respect to their base scheme. That is, if $f \colon X \to Y$ is a morphism of $k$-schemes and $n \in \N \cup \{\infty\}$, then there exists an induced morphism $f_n \colon X_n \to Y_n$. One of the fundamental properties of jet and arc spaces (which follows essentially from the definitions) is that, for \'etale morphisms $f$, the induced morphism $f_n$ is given by base change and is hence again \'etale \cite[Proposition 2.5]{Ish07}. We will mention here a variant of this statement at the level of formal neighborhoods of closed points. As we are dealing with completions of non-Noetherian schemes which differ in subtle ways from the Noetherian case, we include a short alternative argument for the reader's convenience.

For any $k$-scheme $X$ and $x \in X(k)$ the formal neighborhood of $X$ at $x$ is defined as
\[
    \^X_x := \Spf \^{\cO_{X,x}},
\]
where $\^{\cO_{X,x}}$ denotes the completion of $\cO_{X,x}$ with respect to its maximal ideal $\fm_x$, endowed with the inverse limit topology. If $\fm_x$ is not finitely generated, then this topology no longer agrees with the $\fm_x$-adic topology. Moreover, the functor of points of the formal scheme $\^X_x$ is not determined by its values over Artinian rings. Instead, one has to allow general \emph{test rings}, that is, local $k$-algebras $(A,\fm_A)$ with $A/\fm_A = k$ and $\fm_A^d = 0$ for some $d$.

Let $n\in \N \cup \{\infty\}$ and $\alpha \in X_n$ with $x := \alpha(0)$. For any test-ring $(A,\fm_A)$ consider the set $\^{(X_n)}_\alpha(A)$ of $A$-deformations of $\alpha$. This set is in bijection to
\[
    \{ \~\alpha \colon \Spec A[[t]] \to X \mid \~\alpha \equiv \alpha \mod \fm_A \}.
\]
Each such morphism $\~\alpha$ factors through $\Spec \cO_{X,x}$; consider the induced ring map
\[
    \~\alpha^\sharp \colon \cO_{X,x} \to A[[t]].
\]
As $A$ is discrete we have that $A[[t]]$ is a complete local ring and thus $\~\alpha^\sharp$ factors through the completion $\^{\cO_{X,x}}$. To summarize, the formal neighborhood of $\alpha \in X_n$ depends only on the formal neighborhood of $x = \alpha(0) \in X$.

\begin{lemma}[{\cite[Proposition 2.5]{Ish07}}]
\label{l:isom-formal-nbhd}
    Let $X$ be a scheme over $k$ and $n\in \N \cup \{\infty\}$. Let $\alpha \in X_n(k)$ and $x = \alpha(0)$. Assume that there exists a $k$-scheme $Y$, $y\in Y(k)$, and an isomorphism of formal neighborhoods $\^X_x \isom \^Y_y$. Then there exists an induced isomorphism
    \[
        \^{(X_n)}_\alpha \isom \^{(Y_n)}_\beta,
    \]
    where $\beta \in Y_\infty$ is the image of $\alpha$.
\end{lemma}

As an immediate consequence, we obtain the easy implication in \cref{t:main-thm}.

\begin{lemma}
    Let $X$ be a scheme over a field $k$ and $x \in X(k)$. Assume that $\^X_x \isom \^Y_y \times \Spf k[[z]]$, with $Y$ some $k$-scheme and $y\in Y$. For $n\in \N \cup \{\infty\}$ denote by $\alpha_x \in X_n$ and $\beta_y \in Y_n$ the constant $n$-jets resp.\ arcs centered at $x$ and $y$. Then
    \[
        \^{(X_n)}_{\alpha_x} \isom \^{(Y_n)}_{\beta_y} \times \Spf k[[z_0,\ldots,z_n]].
    \]
\end{lemma}

\begin{proof}
    Follows from \cref{l:isom-formal-nbhd} and the fact that, for $X$ smooth of dimension $d$, the jet space $X_n$ is smooth of dimension $d(n+1)$.
\end{proof}

%
%

\section{Derivations and cancellation}
\label{s:derivations}

In this section we collect some results on derivations on the rings underlying formal neighborhoods of non-Noetherian schemes over $k$. For convenience we assign these rings a name.

\begin{definition}
    A \emph{quasi-adic ring} is a local $k$-algebra $(R,\fm)$ which is the completion of a local ring $(A,\fm_A)$ with $A/\fm_A = k$, where we require that if $R$ is Noetherian then so is $A$. We consider $R$ as a topological ring endowed with its inverse limit topology.
\end{definition}

The prefix ``quasi'' here points to the fact that, in case $A$ is not Noetherian, the topology on $R$ can be approximated by the $\fm$-adic one. More precisely, a basis for the open subsets is given by the topological closure of powers $\fm^n$. By definition, any Noetherian quasi-adic ring $R$ is adic.

\begin{example}
    \label{ex:whitney}
    The requirement for $A$ to be Noetherian in case $R$ is Noetherian is motivated by \cite[Example 5.4]{CdFD22}. Consider the following example by Whitney: take $g(t) = \sum_{n\geq 1} g_n t^n$ to be any transcendental series in $\C[[t]]$ and set
    \[
        f := x y (y-x) (y - (3+t)x) (y - (4 + g(t))x) \in \C[[x,y,t]].
    \]
    In \cite{Whi65} it is proven that the Noetherian ring $R := \C[[x,y,t]]/(f)$ is not isomorphic to the completion of a $\C$-algebra essentially of finite type over $\C$. However, in \cite[Example 5.4]{CdFD22} it is shown that $R$ is the completion of the non-Noetherian ring $\C[x,y,t,z_n \mid n \in \Z_{\geq0}]/\fa$, where
    \[
        \fa := (h, z_{n-1} - z_n t - g_n t \mid n \in \Z_{>0})
    \]
    and
    \[
        h := x y (y-x) (y - (3+t)x) (y - (4 + z_0)x).
    \]
\end{example}

Translating the previous definition into the algebraic language, we say that a quasi-adic ring $(R,\fm)$ admits a smooth factor if $\Spf R \isom \cZ \times \Spf k[[z]]$ for some formal scheme $\cZ$. Note that $\cZ$ is necessarily affine, so $R$ has a smooth factor if there exists an admissible ring $S$ (see \cite[0, (7.1.2)]{ega-i}) and an isomorphism of topological rings $R \isom S \cotimes_k k[[z]]$, the completed tensor product of the topological rings $S$ and $k[[z]]$ (see \cite[0, (7.7.6)]{ega-i}).

The main tool in the proof of \cref{t:main-thm} are derivations. Even though we will restrict our attention to the characteristic $0$ case in the next section, here we give the characteristic-free version of the results we make use of. For a more detailed account we refer the reader to \cite{BS19}.

Let $R$ be any $k$-algebra. A \emph{higher derivation} of $R$ is a ring map $d = \sum_{i \geq 0} d_i \colon R\to R[[t]]$ such that $d_0 = \id_R$. The set of all higher derivations of $R$ will be denoted by $\Der_k^\infty(R)$. If $(R,\fm)$ is quasi-adic, then $d \in \Der_k^\infty(R)$ is called \emph{regular} if it is continuous (with respect to the induced topology on $R[[t]] \isom R \cotimes_k k[[t]]$) and there exists $x \in \fm$ with $d_1(x) \in R$ invertible. If $k$ is of characteristic $0$, then any usual derivation $d \in \Der_k(R)$ extends uniquely to a higher derivation.

\begin{lemma}
    \label{l:regular}
    Let $(R,\fm)$ be the completion of $(A,\fm_A)$. Assume that there exists $d \in \Der_k^\infty(R)$ regular. Then there exists $d' \in \Der_k^\infty(R)$ and $x \in \fm_A$ with $d'(x) = x + t$.
\end{lemma}

\begin{proof}
    By assumption there exists $x \in \fm$ such that $d_1(x)$ is invertible. Then there exists an index set $J$ and a surjection of local rings $k[X_j \mid j\in J]_{(X_j \mid j\in J)} \to A$, mapping $X_j$ to $x_j \in \fm_A$. Completing gives rise to a surjection $k[[X_j \mid j\in J]] \to R$, hence there exists $f \in k[[X_j \mid j\in J]]$ such that
    \[
        x = f(x_j \mid j \in J).
    \]
    By continuity of $d$ we have that $d_1(x_j)$ is invertible for some $j \in J$. Then we proceed as in \cite[Lemma 2.1]{BS19}.
\end{proof}

The following is the main result of this section and, in case $R$ is Noetherian, is sometimes known as the Zariski--Lipman--Nagata criterion.

\begin{theorem}
    \label{t:triviality}
    Let $(R,\fm)$ be a quasi-adic ring. The following are equivalent:
    \begin{enumerate}
        \item \label{i1:triviality} There exists a regular $d \in \Der_k^\infty(R)$.
        \item \label{i2:triviality} $R$ admits a smooth factor.
        \item \label{i3:triviality} There exists $(S,\fn)$ quasi-adic with $R \isom S \cotimes_k k[[z]]$.
    \end{enumerate}
\end{theorem}

\begin{proof}
    For the complete proof we refer the reader to \cite[Theorem 1.1]{BS19}; here we just give a brief remark arguing that \eqref{i1:triviality} implies \eqref{i3:triviality} (and not just \eqref{i2:triviality}). Given $d \in \Der_k^\infty$ continuous and $x\in \fm$, it is proven that there exists a subring $S$ of $R$ and a continuous isomorphism
    \[
        \theta \colon S[[Y]] \to R,\: \sum_{i\geq0} f_i Y^i \to \sum_{i\geq0} f_i x^i.
    \]
    Assume that $(R,\fm)$ is the completion of $(A,\fm_A)$. By \cref{l:regular}, we may choose $x\in \fm_A$, and then $\theta$ induces an isomorphism of $S$ with the completion of $A/(x)$.
\end{proof}

\begin{remark}
    \label{r:char0}
    If $k$ has characteristic $0$, then in \eqref{i1:triviality} it is enough to assume the existence of a (usual) derivation $d \in \Der_k(R)$ which is regular, i.e.\, continuous and such that there exists $x\in \fm$ with $d(x)$ invertible.
\end{remark}

\begin{remark}
    In \cite{BS19}, the fact that in \eqref{i3:triviality} of \cref{t:triviality} the ring $S$ is the completion of a local ring is obtained by means of \cite[Lemma 2.3]{BS19}. Note that the proof given there has a gap: in the last sentence it is claimed that $S$ is the completion of $\pi(A)$, which is not true in general. For a counterexample consider the ring $S = k[[x,y,t]]/(f)$ with $f$ as in \eqref{ex:whitney} and let $A = k[x,y,t]_{(x,y,t)}$, $R = k[[x,y,t]]$. The image $\pi(A)$ in $S$ is isomorphic to $A$, so in particular to completion of $\pi(A)$ is not isomorphic to $S$.
\end{remark}

A consequence of \cref{t:triviality} is the following cancellation result. For a proof we refer the reader to \cite{BS19}.

\begin{theorem}[{\cite[Theorem 3.1]{BS19}}]
    \label{t:cancellation}
    Let $(R,\fm)$ and $(S,\fn)$ be quasi-adic rings which do not admit a smooth factor. Assume that there exist index sets $I$ and $J$ and an isomorphism of topological rings
    \[
        \varphi \colon R \cotimes_k k[[x_i \mid i\in I]] \isom S \cotimes_k k[[y_j \mid j \in J]].
    \]
    Then $\varphi$ restricts to an isomorphism $R \isom S$.
\end{theorem}

%
%

\section{Applications to formal neighborhoods of the arc space}
\label{s:formal-nbhd-arc}

Let us move towards the proof of the remaining implication in \cref{t:main-thm}. For that, we assume that $k$ is a field of characteristic $0$. Let $X$ be a scheme over $k$ and $x\in X(k)$. For $n\in \N \cup \{\infty\}$ let $\alpha_x \in X_n$ be the constant $n$-jet resp.\ arc. By \cref{t:triviality,r:char0} it is enough to show that, given a regular derivation $d \in \Der_k(\^{\cO_{X_n,\alpha_x}})$, there exists a regular derivation $d_0 \in \Der_k(\^{\cO_{X,x}})$.

We may assume that $X = \Spec A$ is affine. Then $X_n = \Spec A_n$, where $A_n$ denotes the $n$-th Hasse--Schmidt algebra of $A$ \cite{Voj07}. If $n \in \N$, then the universal $n$-jet $\gamma_n \colon A_n \to A_n[t]/(t^{n+1})$ is given as
\[
    \gamma_n(x) = \sum_{i=0}^n x^{(i)} t^i,
\]
and similarly for the universal arc $\gamma \colon A_\infty \to A_\infty[[t]]$/As described in \cref{s:jets-arcs}, the natural inclusion $p \colon A \to A_n$ corresponding to $\pi_{n,0}$ has a retraction $q \colon A_n \to A$ corresponding to $\iota_n$. If $\fp \subset A$ denotes the prime corresponding to $x$ and $\fp_n \subset A_n$ the prime corresponding to $\alpha_x$, we get a factorization of the identity on $A_\fp$ into
\[
    \xymatrix{A_\fp \ar[r]^-{p_0} & (A_n)_{\fp_n} \ar[r]^-{q_0} & A_\fp.}
\]
Here we use that $\alpha_x$ is the constant arc centererd at $x$. Set $R:= \^{A_\fp}$, $R_n := \^{(A_n)_{\fp_n}}$ and let $\^p$, $\^q$ be the completions of $p_0$ and $q_0$.

\begin{lemma}
    Let $d \in \Der_k(R_n)$ regular. Then there exists $d \in \Der_k(R)$ regular.
\end{lemma}

\begin{proof}
    For simplicity we assume $n\in \N$; the case $n=\infty$ is done similarly. Let $M$ be an $A_n$-module. Set $B_n := A_n[t]/(t^{n+1})$ and consider the map $\gamma_n \colon A \to B_n$ induced by the universal $n$-jet. Then we have isomorphisms (see \cite[Lemma 5.1]{dFD20})
    \begin{align*}
        \Der_k(A_n,M) & \isom \{\phi \in \Hom_k(A_n, A_n \oplus \varepsilon M) \mid \phi \equiv \id_{A_n} \mod \varepsilon\} \\
        & \isom \{\phi \in \Hom_k(A, B_n \oplus \varepsilon (M \otimes_{A_n} B_n)) \mid \phi \equiv \gamma_n \mod \varepsilon\} \\
        & \isom \Der_k(A,M \otimes_{A_n} B_n).
    \end{align*}
    Note that any continuous derivation $d \in \Der_k(R_n)$ is uniquely determined by its restriction in $\Der_k(A_n, R_n)$. Note that $R_n \otimes_{A_n} B_n \isom (R_n)^{n+1}$ as $A_n$-modules. Composing with $\^q \colon R_n \to R$ and by the above $d$ corresponds to a continuous derivation $d_0 \in \Der_k(R,R^{n+1})$. Note that $d$ regular implies that $d_0(x) \neq 0$ for some $x$ in the maximal ideal of $R$. It is straightforward to verify that
    \begin{equation}
    \label{eq:prod-der}
        \Der_k(R,R^{n+1}) \isom \prod_{i=1}^{n+1} \Der_k(R).
    \end{equation}
    and $d_0(x) \neq 0$ implies that one of the components of the image of $d_0$ is a regular derivation in $\Der_k(R)$. This proves the lemma.
\end{proof}

\begin{remark}
    It is not clear what the equivalent of the identity \eqref{eq:prod-der} should be for higher derivations. In fact, this is the main reason to restrict to the characteristic $0$ case.
\end{remark}

\begin{corollary}
    Let $X$ be a scheme of characteristic $0$ and $x\in X(k)$. For $n\in \N \cup \{\infty\}$ let $\alpha_x$ be the constant $n$-jet resp.\ arc. If $\^{(X_n)}_{\alpha_x}$ admits a smooth factor, then so does $\^X_x$.
\end{corollary}

The situation is completely different for nonconstant arcs, as they always admit a smooth factor coming from the universal arc, independent of the characteristic.

\begin{lemma}
    \label{l:nonconstant}
    Let $X$ be a scheme over $k$ and $\alpha \in X_\infty(k)$ nonconstant, that is, not in the image of $\iota_\infty \colon X \to X_\infty$. Then $\^{(X_\infty)}_\alpha$ admits a smooth factor.
\end{lemma}

\begin{proof}
    We may assume $X = \Spec A$ affine; then $X_\infty = \Spec A_\infty$. The universal arc on $X$ is given by 
    \[
        \gamma^\sharp \colon A \to A_\infty[[t]],\: x \mapsto \sum_{i\geq 0} x^{(i)} t^i.
    \]
    Let $\fp_\alpha \subset A_\infty$ be the prime defining $\alpha$. The condition that $\alpha$ is nonconstant is equivalent to the existence of $f \in \fp_{\alpha(0)}$ such that the series $\gamma^\sharp(f)$ is nonzero modulo $\fp_\alpha$, say of order $n$. Let $R_\infty$ be the completion of $A_\infty$ at $\fp_\alpha$. There exists a natural extension of $\gamma^\sharp$ to $A_\infty$, which gives rise to a continuous higher derivation $d \in \Der_k^\infty(R_\infty)$. Note that $d_1(f^{(d)})$ is invertible by construction.
\end{proof}

\begin{question}
    Let $X$ be a scheme over $k$ and $x \in X(k)$ such that $\^X_x$ does not admit a smooth factor. If $\alpha$ is not constant and centered at $x$, then \cref{l:nonconstant} establishes the existence of an isomorphism $\^X_x \isom \^Y_y \times \Spf k[[t]]$. Does $\^Y_y$ admit any smooth factors?
\end{question}

Indeed, an answer to this question would be the crucial step towards a full description of the formal neighborhood of a general degenerate arc.

We now move to several consequences of \cref{t:main-thm}. Motivated by the theorem of Drinfeld, Grinberg and Kazhdan we make the following definition.

\begin{definition}
    Let $X$ be a $k$-scheme and $x\in X(k)$. A \emph{finite formal model} for $x$ is a formal scheme of the form $\^Y_y$, where $Y$ is a scheme of finite type over $k$ and $y\in Y(k)$, such that there exists an isomorphism
    \[
        \^X_x \isom \^Y_y \times \Spf k[[z_i \mid i\in I]].
    \]
    If $\^Y_y$ does not admit a smooth factor, then $\^Y_y$ is called a \emph{minimal formal model} for $x$.
\end{definition}

If $X$ is of finite type over $k$, then any $x\in X(k)$ has a finite formal model (and hence a minimal one). Indeed, in \cite{CH22}, minimal formal models of varieties were given a geometric interpretation in terms of the respective isosingular loci. Similarly, Drinfeld's proof in \cite{Dri} gives an explicit construction of a finite formal model for nondegenerate arcs on a scheme of finite type. On the other hand, by \cref{t:cancellation} a minimal formal model is unique up to isomorphism. In particular, it follows from Cohen's structure theorem that $\^X_x$ is formally smooth if and only if its minimal formal model is given by $\Spf k$.

Another consequence of \cref{t:cancellation} is the following.

\begin{corollary}
    Let $x\in X(k)$. If $\^X_x$ does not admit a smooth factor, then there does not exist a finite formal model for $x$.
\end{corollary}

Putting these together, we obtain the following  description of the formal neighborhoods of arcs \emph{contained in their respective isosingular locus}.

\begin{proposition}
    \label{p:not-dgk}
    Let $X$ be a scheme of finite type over a field of characteristic $0$ and $x \in X(k)$. Let $\^Y_y$ be a minimal formal model of $\^X_x$ and $\alpha \in X_\infty(k)$ centered at $x$ such that the image of $\alpha$ in $Y_\infty$ equals the constant arc $\beta_y$. Then
    \[
        \^{(X_\infty)}_{\alpha} \isom \^{(Y_\infty)}_{\beta_y} \times \Spf k[[z_i \mid i\in I]]
    \]
    and $\^{(Y_\infty)}_{\beta_y}$ does not admit a smooth factor. In particular, if $X$ is not smooth at $x$, then there does not exist a finite formal model for $\^{(X_\infty)}_{\alpha}$.
\end{proposition}

Finally we want to mention in \cref{p:comparison} a result comparing the formal neighborhood of an arbitrary arc $\alpha$ to that of the constant arc centered at $\alpha(0)$. It makes precise the idea that the arc space has its worst singularities at constant arcs centered in the singular locus.

For convenience's sake we introduce the following notion. A morphism $f \colon \cX \to \cY$ of formal $k$-schemes is called \emph{efficient} if for all $(A,\fm_A)$, with $A$ a $k$-algebra and $\fm_A \subset A$ some (not necessarily maximal) ideal with $\fm_A^2 = 0$, we have that the map on $A$-points $f_A \colon \cX(A) \to \cY(A)$ is a bijection. Note that here we consider $A$ as a topological ring with the discrete topology (which equals the $\fm_A$-adic one). We note that this definition is compatible with \cite[Definition 3.6]{CdFD22}:

\begin{lemma}
    Let $f \colon \Spf S \to \Spf R$ with $(R,\fm)$, $(S,\fn)$ completions of local $k$-algebras satisfying $R/\fm = S/\fn = k$. Then $f$ is efficient if and only if $f^\sharp \colon R \to S$ is surjective and induces an isomorphism of continuous cotangent spaces, that is,
    \[
        \fm/\^{\fm^2} \isom \fn/\^{\fn^2}.
    \]
\end{lemma}

The proof is straightforward.

\begin{lemma}
    \label{l:efficient}
    Let $f \colon X \to Y$ a morphism of $k$-schemes and let $x \in X(k)$, $y = f(x) \in Y(k)$. Let $\alpha \in X_\infty(k)$ with $\alpha(0) = x$, set $\beta = f_\infty(\alpha)$ and as before denote by $\alpha_x \in X_\infty$, $\beta_y \in Y_\infty$ the constant arcs centered at $x$ and $y$.
    \begin{enumerate}
        \item If $\^f \colon \^X_x \to \^Y_y$ is a closed immersion, then so is $\^{(X_\infty)}_\alpha \to \^{(Y_\infty)}_\beta$.
        \item If $\^f \colon \^X_x \to \^Y_y$ is efficient, then so is $\^{(X_\infty)}_{\alpha_x} \to \^{(Y_\infty)}_{\beta_y}$.
    \end{enumerate}
\end{lemma}

\begin{proof}
    The first assertion is immediate, so let us assume that $\^f$ is efficient. Let $(A,\fm_A)$ be a $k$-algebra with $\fm_A^2 = 0$. Then any $A$-valued point $\~\alpha$ of $\^{(X_\infty)}_{\alpha_x}$ corresponds to a morphism
    \[
        \~\alpha \colon \Spec A[[t]] \to X
    \]
    which modulo $\fm_A$ factors through $\alpha_x \colon \Spec k[[t]] \to X$. Hence $\~\alpha \in \^X_x(A[[t]])$, where we consider $(A[[t]],\fm_A[[t]])$ with the discrete topology. This proves the claim.
\end{proof}

The aforementioned comparison result now follows:

\begin{theorem}
    \label{p:comparison}
    Let $X$ be a scheme of finite type over a field $k$ of characteristic $0$. Let $\alpha \in X_\infty(k)$ and $x = \alpha(0)$. Let $\^Y_y$ be a minimal formal model for $\^X_x$. Then there exist index sets $J \subset I$ and a commutative diagram of formal schemes
    \[
        \xymatrix{\Spf k[[z_i \mid i\in I]]  \ar[rd]^-{p_0} & \^{(X_\infty)}_\alpha \ar[d]^-{p_1} \ar[l]_-{i_1} \\
        \^{(X_\infty)}_{\alpha_x} \ar[u]_-{i_2} \ar[r]_-{p_2} & \Spf k[[z_j \mid j \in J]],}
    \]
    such that $i_1$, $i_2$ are closed immersions and $p_0$, $p_1$, $p_2$ are natural projections induced from a product decomposition. Moreover:
    \begin{enumerate}
        \item $i_2$ is efficient, and
        \item $p_2$ equals $\pr_2 \colon \^{(Y_\infty)}_{\beta_y} \times \Spf k[[z_j \mid j \in J]] \to \Spf k[[z_j \mid j\in J]]$ and $\^{(Y_\infty)}_{\beta_y}$ does not admit a smooth factor.
    \end{enumerate}
\end{theorem}

\begin{remark}
    If $x \in X_\sm$, then $I = J$, $\^Y_y = \Spf k$ and $i_1, i_2, p_0, p_1, p_2$ can all be chosen to be isomorphisms.
\end{remark}

\begin{remark}
    \label{r:comparison}
    Let $\^{(X_\infty)}_\alpha \isom \^Z_z \times \Spf k[[z_i \mid i\in J]]$. Then \cref{p:comparison} guarantees the existence of closed immersions
    \[
        j_1 \colon \^Z_z \to \Spf k[[z_i \mid i\in I \setminus J]], \: j_2 \colon \^{(Y_\infty)}_{\beta_y} \to \Spf k[[z_i \mid i\in I \setminus J]]
    \]
    with $j_2$ being efficient.
\end{remark}

\begin{proof}[Proof of \cref{p:comparison}]
    By \cref{l:isom-formal-nbhd} we may assume $X = Y \times \A^d$, $x = (y,0)$ and by\footnote{One may avoid the use of \cite[Theorem 3.15]{CdFD22} (and so the assumption that $k$ is infinite) by considering the arc space of the formal scheme $\^X_x$ instead.} \cite[Theorem 3.15]{CdFD22} we have $Y \subset \A^n$ with $n = \embdim (\cO_{Y,y})$. Note that $\^f \colon \^{(\A^{n+d})} \to \^X_x$ is efficient and, since $\A^{n+d}$ is smooth, any two $k$-arcs on it have isomorphic formal neighborhoods. Set $J := 0 \times \Z_{\geq0}^d \subset \Z_{\geq0}^{n+d} =: I$ and consider all maps induced by $\^f$. The assertion then follows from \cref{t:main-thm,l:efficient}.
\end{proof}

\begin{remark}
    Compare \cref{p:comparison} to a similar result from \cite{CdFD22}. There it was proven that, for any arc $\alpha$ on a variety $X$ over a perfect field $k$, the condition $\alpha \notin (X \setminus X_\sm)_\infty$ is equivalent to
    \[
        \embcodim (\cO_{X_\infty,\alpha}) < \infty.
    \]
    See \cite[Definition 6.1]{CdFD22} for the definition of $\embcodim(R)$ for non-Noetherian local rings $(R,\fm)$. In particular, if $\alpha_x$ denotes the constant arc centered at $x \in X \setminus X_\sm(k)$, then the embedding codimension of $\alpha_x$ is infinite. On the other hand, for any nondegenerate $\alpha \in X_\infty(k)$ centered at $x$ we have
    \[
        \embcodim (\cO_{X_\infty,\alpha}) < \infty = \embcodim (\cO_{X_\infty,\alpha_x}).
    \]
    As mentioned in \cref{r:comparison}, \cref{p:comparison} implies that, for any $\alpha \in X_\infty(k)$ centered at $x$, we have
    \[
        \embdim (\cO_{X_\infty,\alpha}) \leq \embdim (\cO_{X_\infty,\alpha_x}).
    \]
    However, there is no similar bound for the embedding codimensions that follows as an immediate consequence of \cref{p:comparison}. For example, there exists a principal ideal $(f) \subset k[[x_i \mid i \in \N]]$ such that the quotient $A$ (which obviously satisfies $\embcodim(A) = 1$) admits no smooth factors, see \cite[Remark 5.5]{CdFD22}.
\end{remark}

\printbibliography[
heading=bibintoc,
title={References}
]

\end{document}